\documentclass{article}
\usepackage{amssymb,amsmath}

\newtheorem{theorem}{Theorem}[section]

\newtheorem{proof}{Proof}

\def\hs{\hskip0.04cm'}
\def\dz{,\kern-0.1em ,}
\def\dd#1{{#1\kern-0.4em\char"16\kern-0.1em}}
\def\D#1{{\raise0.2ex\hbox{-}\kern-0.42em #1}}

\def\la{\lambda}
\def\d{\delta}

\def\a{\alpha}
\def\b{\beta}
\def\G{{\rm\Gamma}}
\def\g{\gamma}

\def\lesl{\leqslant}

\def\z{\zeta}

\def\Re{{\rm Re\,}}

\def\rmd{\mathrm{d}}

\newcommand{\arctg}{\mathop{\rm arctg}}

\title{On closed forms of some trigonometric series}

\author{Slobodan B. Tri\v ckovi\'c\\
\small University of Ni\v s, Serbia, Department of Mathematics, e-mail: sbt@junis.ni.ac.rs\\
Miomir S. Stankovi\'c\\
\small Mathematical Institute of Serbian Academy of\\
\small Sciences and Arts, Belgrade, Serbia, e-mail: miomir.stankovic@gmail.com}

\date{\today}

\begin{document}

\maketitle

\begin{abstract}
We have derived alternative closed-form formulas for the trigonometric series
over sine or cosine functions when the immediate replacement of the parameter appearing in
the denominator with a positive integer gives rise to a singularity. By applying
the Choi-Srivastava theorem, we reduce these trigonometric series to expressions over Hurwitz's
zeta function derivative.
\end{abstract}


Subject class: Primary 11M41; Secondary 33B15.

\markright{On closed forms of some trigonometric series}

\section{Preliminaries}\label{s1}

All particular cases of the general summation formula for the trigonometric series
\begin{equation}\label{trig}
T_\a^f=\sum_{n=1}^\infty\frac{(s)^{n-1}f((an-b)x)}{(an-b)^\a},\quad\a>0,
\end{equation}
we derived in \cite{itsf-19-6} and expressed \eqref{trig}
via a power series
\begin{equation}\label{cosseries}
T_\a^f=\frac{c\pi x^{\a-1}}{2\G(\a)f(\frac{\pi\a}2)}
+\sum_{k=0}^\infty\frac{(-1)^k F(\a-2k-\d)}{(2k+\d)!}x^{2k+\d},
\end{equation}
and one can obtain each of them by looking up in Table I, p. \pageref{table1} by taking the corresponding parameters,
where $F$ is Riemann's $\z$ function
initially defined by the series
\begin{equation}\label{z}
\z(s)=\sum_{k=1}^\infty\frac1{k^s},\quad\Re s>1.
\end{equation}
or Dirichlet $\eta,\la,\b$ functions defined as follows
\begin{equation}\label{eta-lambda-beta}
\eta(s)=\sum_{n=1}^\infty\frac{(-1)^{n-1}}{n^s},\quad\la(s)=\sum_{n=1}^\infty\frac1{(2n-1)^s},
\quad\b(s)=\sum_{n=1}^\infty\frac{(-1)^{n-1}}{(2n-1)^s},
\end{equation}

Riemann's $\z$ function satisfies the functional equation \cite{apostol}
\begin{equation}\label{z-fun-eq}
\z(1-s)=\frac{2\z(s)}{(2\pi)^s}\G(s)\cos\frac{\pi s}2.
\end{equation}
where $\G(s)$ is the gamma function
\begin{equation}\label{gamma-f}
\G(s)=\int_0^\infty x^{s-1}e^{-x}\,\rmd x\quad(\Re s>0)
\end{equation}
introduced by Euler. Rewriting the integral \eqref{gamma-f} as follows
\begin{align*}
\G(s)=\int_0^\infty x^{s-1}e^{-x}\rmd x
&=\sum_{n=0}^\infty\frac{(-1)^n}{n!}\int_0^1 x^{s+n-1}\,\rmd x
+\int_1^\infty x^{s-1}e^{-x}\rmd x\\
&=\sum_{n=0}^\infty\frac{(-1)^n}{n!}\frac1{s+n}+\int_1^\infty x^{s-1}e^{-x}\rmd x
\end{align*}
provides the analytic continuation of the gamma function for all complex numbers,
except for integers less than or equal to zero. So, the functional equation \eqref{z-fun-eq}
extends $\z(s)$ to the whole complex plane except for $s=1$, where it has a pole, since $s=0$
appears as a singularity of the gamma function. For $s=2n+1,\,n\in\mathbb N$, in \eqref{z-fun-eq},
we find
\begin{equation}\label{zeta2n-zero}
\z(-2n)=\frac{2(2n)!}{(2\pi)^{2n}}\z(2n+1)\cos\frac\pi2(2n+1)=0.
\end{equation}


The integral representation of the $\eta$ function
\begin{equation}\label{eta-int}
\eta(s)=\sum_{n=1}^\infty\frac{(-1)^{n-1}}{n^s}=\frac1{\G(s)}\int_0^\infty\frac{x^{s-1}}{e^x+1}\,\rmd x
\end{equation}
define it as analytical function for $\Re s>0$, but by using its functional equation
\begin{equation}\label{eta-ext}
\eta(-s)=\frac{2^{s+1}-1}{2^{s}-1}\frac s{\pi^{s+1}}\sin\frac{\pi s}2\G(s)\,\eta(s+1)
\end{equation}
and by Euler's reflection formula \cite{olver}
\begin{equation}\label{euler-refl}
\G(1-s)\G(s)=\frac\pi{\sin\pi s},\quad s\notin\mathbb Z,
\end{equation}
the eta function extends to the whole complex plane, and \eqref{eta-ext} immediately gives rise
to $\eta(-2n)=0$, $n\in\mathbb N$, which we can also calculate from its connection to the zeta function.


Given \eqref{eta-lambda-beta}, we conclude that Dirichlet's lambda function,
as the zeta function, is analytical for all complex numbers except for $s=1$, where it has a pole.
In addition, $\la(-2n)=0$, $n\in\mathbb N$.


The integral representation of the $\b$ function
\begin{equation}\label{beta-int}
\b(s)=\sum_{n=1}^\infty\frac{(-1)^{n-1}}{(2n-1)^s}
=\frac1{\G(s)}\int_0^\infty\frac{x^{s-1}e^{-x}}{1+e^{-2x}}\rmd x,
\end{equation}
defines it as an analytical function for $\Re s>0$, but by means of the functional equations
\begin{equation}\label{beta-ext}
\b(1-s)=\Big(\frac2\pi\Big)^s\sin\frac{\pi s}2\G(z)\b(s)
\end{equation}
$\b$ extends to the whole complex plane. By virtue of \eqref{euler-refl},
the equation \eqref{beta-ext} can be rewritten as follows
\begin{equation}\label{beta-z}
\b(s)=\Big(\frac\pi2\Big)^{s-1}\b(1-s)\G(1-s)\cos\frac{\pi s}2,
\end{equation}
whence we find $\b(-2n+1)=0$, $n\in\mathbb N$.


Thus, owing to 
$$
\z(-2n)=\eta(-2n)=\la(-2n)=\b(-2n+1)=0,\quad n\in\mathbb N
$$
the right-hand side series in \eqref{cosseries} truncates. Thereby we obtained
in \cite{itsf-19-6} one type of closed-form formulas. So, for $\a=2m+p-1$ where
$m\in\mathbb N$ and $p=0$ or $p=1$, we have all these cases comprised by the general
formula \cite{itsf-19-6}
\begin{align*}
T_{2m+p-1}^f=\sum_{n=1}^\infty&\frac{(s)^{n-1}\,f(an-b)x)}{(an-b)^{2m+p-1}}
=\frac{c\pi x^{2m+p-2}}{2(2m+p-2)!f(m\pi+\frac\pi2(p-1))}\\
&+\sum_{k=0}^{m+p-1}\frac{(-1)^k F(2m+p-1-2k-\d)}{(2k+\d)!}x^{2k+\d},
\end{align*}
and we can obtain each of them from Table I by choosing the corresponding parameters.

{\tabcolsep 4pt
\begin{center}
\begin{tabular}{|c|c|c|c|c|c|c|c|c|} \multicolumn{9}{c}
{Table I: Particular and closed-form cases of \eqref{cosseries}}\label{cl-form}\\ \hline
 $a$ & $b$ & $s$ & $c$ & $F$ & $f$ & $\d$ & $p$  & Convergence region\\ \hline
   &   &     &     &      & $\sin$ & $1$ & $0$  & \\ \cline{6-8}
   &   & \raisebox{1.5ex}[0pt]{$1$} & \raisebox{1.5ex}[0pt]{$1$} &
   \raisebox{1.5ex}[0pt]{$\z$} & $\cos$ & $0$ & $1$  & \raisebox{1.5ex}[0pt]{$0<x<2\pi$}\\ \cline{3-9}
   &   &     &     &      & $\sin$ & $1$ & $0$  & \\ \cline{6-8}
 \raisebox{4.5ex}[0pt]{$1$} & \raisebox{4.5ex}[0pt]{$0$} & \raisebox{1.5ex}[0pt]{$-1$} & \raisebox{1.5ex}[0pt]{$0$} &
 \raisebox{1.5ex}[0pt]{$\eta$} & $\cos$ & $0$ & $1$  & \raisebox{1.5ex}[0pt]{$-\pi<x<\pi$}\\ \hline
 &   &     &     &      & $\sin$ & $1$ & $0$  & \\ \cline{6-8}
   &   &  \raisebox{1.5ex}[0pt]{$1$} & \raisebox{1.5ex}[0pt]{$\frac 12$} &
   \raisebox{1.5ex}[0pt]{$\la$} & $\cos$ & $0$ & $1$ & \raisebox{1.5ex}[0pt]{$0<x<\pi$}\\ \cline{3-9}
   &   &     &     &      & $\sin$ & $1$ & $1$  & \\ \cline{6-8}
 \raisebox{4.5ex}[0pt]{$2$} & \raisebox{4.5ex}[0pt]{$1$} & \raisebox{1.5ex}[0pt]{$-1$} &
  \raisebox{1.5ex}[0pt]{$0$} & \raisebox{1.5ex}[0pt]{$\b$} &
 $\cos$ & $0$ & $0$ & \raisebox{1.5ex}[0pt]{$-\frac\pi 2<x<\frac\pi2$}\\ \hline
\end{tabular}
\end{center}}\label{table1}

\section{Alternative closed-form formulas}\label{s2}

For $a=1,\ b=0,\ s=1$ in \eqref{trig}, and putting $f=\sin$, then $f=\cos$, we have
\begin{equation}\label{polylogarithm1}
\begin{split}
\sum_{n=1}^\infty\frac{\sin nx}{n^\a}
&=\frac{\pi x^{\a-1}}{2\G(\a)\sin\frac\pi2\a}
+\sum_{k=0}^\infty\frac{(-1)^k\z(\a-2k-1)}{(2k+1)!}\,x^{2k+1},\\
\sum_{n=1}^\infty\frac{\cos nx}{n^\a}
&=\frac{\pi x^{\a-1}}{2\G(\a)\cos\frac\pi2\a}
+\sum_{k=0}^\infty\frac{(-1)^k\z(\a-2k)}{(2k)!}\,x^{2k},\quad 0<x<2\pi,
\end{split}
\end{equation}
where $\z$ presents Riemann's zeta function.

However, by taking limits $\a\to2m$ and $\a\to2m-1$ in the first and
the second formula of \eqref{polylogarithm1} respectively, one encounters singularities of the $\z$ function,
so we have to act differently.

\begin{theorem} Letting $\a\to2m$ in the first formula of \eqref{polylogarithm1},
one brings the sine series in closed form
\begin{equation}\label{lim-value1}
\sum_{n=1}^\infty\frac{\sin nx}{n^{2m}}
=\frac{(-1)^{m}(2\pi)^{2m-1}}{(2m-1)!}\Big(\z{\hskip0.2mm}'\Big(1-2m,1-\frac x{2\pi}\Big)
-\z{\hskip0.2mm}'\Big(1-2m,\frac x{2\pi}\Big)\Big),
\end{equation}
where $0<x<2\pi$. The sine series presents the Clausen function $\textrm{\rm Cl}_{2m}$ \rm{\cite{itsf-34-6}}.
\end{theorem}

\begin{proof} We are not allowed to replace $\a$ with $2m$ in the first formula of \eqref{polylogarithm1}
since $\sin\frac\pi2\a=0$ and for $k=m-1$ the $\z$ function has the pole, so we encounter singularities.

We make use of the Hurwitz zeta formula \cite{apostol}, rewriting it in a more suitable form
\begin{equation}\label{hurwitz-formula}
\z(1-s,a)=\frac{2\,\G(s)}{(2\pi)^s}\sum_{n=1}^\infty\frac{1}{n^s}\cos(\tfrac{\pi s}2-2n\pi a),
\end{equation}
where Hurwitz's zeta function initially is defined for $\Re s>1$ by the series
\begin{equation}\label{hurwitz-sum}
\z(s,a)=\sum_{k=0}^\infty\frac1{(k+a)^s},\quad0<a\lesl 1,
\end{equation}
but its integral representation
\begin{equation*}
\z(s,a)=\int_0^\infty t^{s-1}\sum_{k=0}^\infty e^{-(k+a)t}\,\rmd t
=\frac1{\G(z)}\int_0^\infty\frac{t^{s-1}e^{-at}}{1-e^{-t}}\,\rmd t
\end{equation*}
enables us to extend it for all complex numbers, excluding $s=1$.

In the first derivative of the formula \eqref{hurwitz-formula} for $s=2m$ and $a=\frac x{2\pi}$,
we have
\begin{multline*}
-\z\hs\Big(1-2m,\frac x{2\pi}\Big)=\frac{2(-1)^{m}\G(2m)}{(2\pi)^{2m}}
\big(\psi(2m)-\log 2\pi\big)\sum_{n=1}^\infty\frac{\cos nx}{n^{2m}}\\
-\frac{2(-1)^{m}\G(2m)}{(2\pi)^{2m}}\sum_{n=1}^\infty\frac{\cos nx}{n^{2m}}\log n
-\frac{\pi(-1)^{m-1}\G(2m)}{(2\pi)^{2m}}\sum_{n=1}^\infty\frac{\sin nx}{n^{2m}}.
\end{multline*}
For $a=1-\frac x{2\pi}$, we have
\begin{multline*}
-\z\hs\Big(1-2m,1-\frac x{2\pi}\Big)=\frac{2(-1)^{m}\G(2m)}{(2\pi)^{2m}}
\big(\psi(2m)-\log 2\pi\big)\sum_{n=1}^\infty\frac{\cos nx}{n^{2m}}\\
-\frac{2(-1)^{m}\G(2m)}{(2\pi)^{2m}}\sum_{n=1}^\infty\frac{\cos nx}{n^{2m}}\log n
-\frac{\pi(-1)^{m}\G(2m)}{(2\pi)^{2m}}\sum_{n=1}^\infty\frac{\sin nx}{n^{2m}}.
\end{multline*}
By subtracting these equalities, we get
$$
\z\hs\Big(1-2m,1-\frac x{2\pi}\Big)-\z\hs\Big(1-2m,\frac x{2\pi}\Big)
=\frac{(-1)^{m}(2m-1)!}{(2\pi)^{2m-1}}\sum_{n=1}^\infty\frac{\sin nx}{n^{2m}}.
$$
Hence, we obtain the sum of the sine series. \end{proof}

\begin{theorem} Letting $\a\to2m-1$ in the second formula of \eqref{polylogarithm1},
one brings the cosine series in closed form
\begin{equation}\label{lim-value1-2}
\sum_{n=1}^\infty\frac{\cos nx}{n^{2m-1}}
=\frac{(-1)^{m-1}(2\pi)^{2m-2}}{(2m-2)!}\Big(\z{\hskip0.2mm}'\Big(2-2m,1-\frac x{2\pi}\Big)
+\z{\hskip0.2mm}'\Big(2-2m,\frac x{2\pi}\Big)\Big),
\end{equation}
where $0<x<2\pi$. The cosine series presents the Clausen function $\textrm{\rm Cl}_{2m-1}$ \rm{\cite{itsf-34-6}}.
\end{theorem}

\begin{proof} Similarly, one cannot immediately replace $\a$ with $2m-1$, so it is necessary to perform
a different method. Taking the first derivative of the formula \eqref{hurwitz-formula} with
respect to $s$ for $a=\frac x{2\pi}$, we have
\begin{multline*}
-\z\hs\Big(1-s,\frac x{2\pi}\Big)=\frac{2\,\G(s)}{(2\pi)^s}
\big(\psi(s)-\log 2\pi\big)\sum_{n=1}^\infty\frac{\cos\big(\frac{\pi s}2-nx\big)}{n^s}\\
-\frac{2\,\G(s)}{(2\pi)^s}\sum_{n=1}^\infty\frac{\cos\big(\frac{\pi s}2-nx\big)}{n^s}\log n
-\frac{\pi\G(s)}{(2\pi)^s}\sum_{n=1}^\infty\frac{\sin\big(\frac{\pi s}2-nx\big)}{n^s}.
\end{multline*}
Setting here $s=2m-1$, we find
\begin{multline*}
-\z\hs\Big(2-2m,\frac x{2\pi}\Big)=\frac{2(-1)^{m-1}\G(2m-1)}{(2\pi)^{2m-1}}
\big(\psi(2m-1)-\log 2\pi\big)\sum_{n=1}^\infty\frac{\sin nx}{n^{2m-1}}\\
-\frac{2(-1)^{m-1}\G(2m-1)}{(2\pi)^{2m-1}}\sum_{n=1}^\infty\frac{\sin nx}{n^{2m-1}}\log n
-\frac{\pi(-1)^{m-1}\G(2m-1)}{(2\pi)^{2m-1}}\sum_{n=1}^\infty\frac{\cos nx}{n^{2m-1}}.
\end{multline*}
For $a=1-\frac x{2\pi}$, we have
\begin{multline*}
-\z\hs\Big(2-2m,1-\frac x{2\pi}\Big)=\frac{2(-1)^{m}\G(2m-1)}{(2\pi)^{2m-1}}
\big(\psi(2m-1)-\log 2\pi\big)\sum_{n=1}^\infty\frac{\sin nx}{n^{2m-1}}\\
-\frac{2(-1)^{m}\G(2m-1)}{(2\pi)^{2m-1}}\sum_{n=1}^\infty\frac{\sin nx}{n^{2m-1}}\log n
-\frac{\pi(-1)^{m-1}\G(2m-1)}{(2\pi)^{2m-1}}\sum_{n=1}^\infty\frac{\cos nx}{n^{2m-1}}.
\end{multline*}
By adding these equalities, we get
$$
\z\hs\Big(2-2m,1-\frac x{2\pi}\Big)+\z\hs\Big(2-2m,\frac x{2\pi}\Big)
=\frac{(-1)^{m-1}(2m-2)!}{(2\pi)^{2m-2}}\sum_{n=1}^\infty\frac{\cos nx}{n^{2m-1}}.
$$
Thence we obtain the sum of the cosine series. \end{proof}

For $a=1,\ b=0,\ s=-1$ and putting $f=\sin$ in \eqref{trig},
based on \eqref{polylogarithm1} and using its relation to the zeta function, $\eta(s)=(1-2^{1-s})\z(s)$,
we obtain the alternating series over the sines and cosines expressed in terms of the Dirichlet eta function
\begin{align}
\sum_{n=1}^\infty\frac{(-1)^{n-1}\sin nx}{n^\a}&=\sum_{n=1}^\infty\frac{(-1)^{n-1}\sin nx}{n^\a}
-\frac1{2^{\a-1}}\sum_{n=1}^\infty\frac{\sin2nx}{n^\a}\label{polylogarithm1-2-1}\\
&=\sum_{k=0}^\infty\frac{(-1)^k\eta(\a-2k-1)}{(2k+1)!}\,x^{2k+1},\quad-\pi<x<\pi.\notag
\end{align}

Similarly, for $a=1,\ b=0,\ s=-1$ and putting $f=\cos$ in \eqref{trig}, we find
\begin{equation}\label{polylogarithm1-1-2}
\sum_{n=1}^\infty\frac{(-1)^{n-1}\cos nx}{n^\a}
=\sum_{k=0}^\infty\frac{(-1)^k\eta(\a-2k)}{(2k)!}\,x^{2k},\quad-\pi<x<\pi.
\end{equation}

\begin{theorem} For $\a=2m$, the series \eqref{polylogarithm1-2-1} takes the closed form
\begin{equation}\label{cf-sin-al}
\begin{split}
\sum_{n=1}^\infty&\frac{(-1)^{n-1}\sin nx}{n^{2m}}
=\frac{(-1)^{m}\pi^{2m-1}}{(2m-1)!}
\bigg(2^{2m-1}\z{\hskip0.15mm}'\Big(1-2m,1-\frac x{2\pi}\Big)\\
&-2^{2m-1}\z{\hskip0.15mm}'\Big(1-2m,\frac x{2\pi}\Big)
-\z{\hskip0.15mm}'\Big(1-2m,1-\frac x\pi\Big)+\z{\hskip0.15mm}'\Big(1-2m,\frac x\pi\Big)\bigg),
\end{split}
\end{equation}
\end{theorem}

\begin{proof} These replacements are legitimate because we do not encounter singularities
because of \eqref{eta-ext} the function $\eta$ being analytic in the whole complex plane.
So we easily calculate
$$
\eta(1)=\lim_{s\to1}\eta(s)=\lim_{s\to1}(1-2^{1-s})\z(s)=\lim_{s\to1}\frac{1-2^{1-s}}{s-1}(s-1)\z(s)=\log2.
$$
Further, from the first formula of \eqref{polylogarithm1-1-2}, we have
\begin{equation*}
\sum_{n=1}^\infty\frac{(-1)^{n-1}\sin nx}{n^{2m}}
=\sum_{k=0}^\infty\frac{(-1)^k\eta(2m-2k-1)}{(2k+1)!}x^{2k+1},
\end{equation*}
and split the right-hand side series as follows
\begin{multline}\label{log2}
\sum_{k=0}^{m-2}\frac{(-1)^k\eta(2m-2k-1)}{(2k+1)!}x^{2k+1}+\frac{(-1)^{m-1}x^{2m-1}}{(2m-1)!}\log2\\
+\sum_{k=m}^\infty\frac{(-1)^k\eta(2m-2k-1)}{(2k+1)!}x^{2k+1}.
\end{multline}
Taking account of $\eta(2m-2k-1)=(1-2^{1-(2m-2k-1)})\z(2m-2k-1)$, we can express the last right-hand side
series in terms of the zeta function as a difference of two series
\begin{equation}\label{k_m}
\sum_{k=m}^\infty\frac{(-1)^k\z(2m-2k-1)}{(2k+1)!}x^{2k+1}
-\sum_{k=m}^\infty\frac{(-1)^k\z(2m-2k-1)}{2^{2m-2k-2}(2k+1)!}x^{2k+1}.
\end{equation}

If we set $z=2n$ in \eqref{z-fun-eq}, we can determine values of Riemann's zeta
function at negative odd integers
\begin{equation}\label{z1-2n}
\z(1-2n)=(-1)^n\frac{2(2n-1)!\z(2n)}{(2\pi)^{2n}}.
\end{equation}
On the other hand, integration by parts of the integral \eqref{gamma-f} gives rise to the basic relation
\begin{equation}\label{gamma-der}
\G(s+1)=s\,\G(s),
\end{equation}
which can extend for arbitrary $n\in\mathbb N$
\begin{equation*}
\G(s+n)=(s+n-1)(s+n-2)\cdots s\,\G(s).
\end{equation*}
The last formula one uses to express the Pochhammer symbol
\begin{equation}\label{pochhammer}
(s)_n=s(s+1)(s+2)\cdots(s+n-1)=\frac{\G(s+n)}{\G(s)}.
\end{equation}

By shifting indices in both series of \eqref{k_m} and applying \eqref{z1-2n} and \eqref{pochhammer},
the difference \eqref{k_m} becomes
\begin{equation}\label{diff-zeta-2k}
2(-1)^{m-1}x^{2m-1}\bigg(\sum_{k=1}^\infty\frac{\z(2k)}{(2k)_{2m}}\Big(\frac x{2\pi}\Big)^{2k}
-\sum_{k=1}^\infty\frac{\z(2k)}{(2k)_{2m}}\Big(\frac x\pi\Big)^{2k}\bigg).
\end{equation}
where $(2k)_{2m}$ denotes Pochhammer's symbol given by \eqref{pochhammer}. We refer now to
Theorem in \cite[p.~419]{choi}, which we state here in a slightly modified form.
\begin{theorem}\label{th-choi-srivastava} For every non-negative integer $n$ there holds
\begin{align}
&\hskip-0.2cm\sum_{k=2}^\infty\frac{\z(k,a)}{(k)_{n+1}}t^{n+k}=\frac{(-1)^n}{n!}
\Big(\z{\hskip0.15mm}'(-n,a-t)-\z{\hskip0.15mm}'(-n,a)
+\sum_{k=1}^n(-t)^k\binom{n}{k}\times\notag\\
&\big(\z(k-n,a)(H_n-H_{n-k})-\z{\hskip0.15mm}'(k-n,a)\big)\Big)
+\big(H_n+\psi(a)\big)\frac{t^{n+1}}{(n+1)!},\label{choi-srivastava}
\end{align}
with $|t|<|a|,\,n\in\mathbb N_0$.
\end{theorem}
In the formula \eqref{choi-srivastava}, $H_n$ is the $n$th harmonic number given as
the sum of reciprocal values of the first $n$ positive integers,
and the function $\psi$ (also known as the digamma)
is the first logarithmic derivative of the gamma function \cite{bateman}.
They are connected by
\begin{equation}\label{psi-hn}
H_{n-1}=\psi(n)+\g,\quad n\in\mathbb N,\quad H_0=0.
\end{equation}

For the first series of \eqref{diff-zeta-2k}, we set in \eqref{choi-srivastava}
$a=1$, $n=2m-1$, $t=\frac x{2\pi}$, then again $a=1$, $n=2m-1$, $t=-\frac x{2\pi}$, and subtract
the second equality from the first. This way, we obtain
\begin{align*}
&\hskip-0.3cm2(-1)^{m-1}x^{2m-1}\sum_{k=1}^\infty\frac{\z(2k)}{(2k)_{2m}}\Big(\frac x{2\pi}\Big)^{2k}
=\frac{(-1)^m(2\pi)^{2m-1}}{(2m-1)!}\times\\
&\bigg(\z\hs\Big(-2m+1,1-\frac x{2\pi}\Big)-\z\hs\Big(-2m+1,1+\frac x{2\pi}\Big)
-2\sum_{k=1}^{m}\binom{2m-1}{2k-1}\times\notag\\
&\Big(\frac x{2\pi}\Big)^{2k-1}\big((H_{2m-1}-H_{2m-2k})\z(2k-2m)-\z\hs(2k-2m)\big)\bigg).\notag
\end{align*}
Knowing that $\z(2k-2m)=0$ for $k=1,\dots,m-1$ and $\z(0)=-1/2$, $H_0=0$, the right-hand side sum becomes
$$
\Big(\frac x{2\pi}\Big)^{2m-1}H_{2m-1}+2\sum_{k=1}^{m}\binom{2m-1}{2k-1}\Big(\frac x{2\pi}\Big)^{2k-1}\z\hs(2k-2m).
$$

From \eqref{hurwitz-sum} we immediately derive its basic property
\begin{equation}\label{hurwitz-a-a+1}
\z(s,a)=\frac1{a^s}+\frac1{(1+a)^s}+\frac1{(2+a)^s}+\cdots=a^{-s}+\z(s,a+1).
\end{equation}
Taking the derivative with respect to $s$ on both sides of \eqref{hurwitz-a-a+1}, putting afterwards
there $a=\frac x{2\pi}$ and $s=1-2m$, we find
$$
-\z\hs\Big(1-2m,1+\frac x{2\pi}\Big)=-\Big(\frac x{2\pi}\Big)^{2m-1}\log\frac x{2\pi}-\z\hs\Big(1-2m,\frac x{2\pi}\Big).
$$

We obtain the same structure for the second series, with the sole difference that instead of $\frac x{2\pi}$,
there appears $\frac x\pi$. The subtraction yields the sum of \eqref{diff-zeta-2k}, which is the sum of the right-hand
side series in \eqref{log2}. So for \eqref{diff-zeta-2k} we find
\begin{equation}\label{subtr}
\begin{split}
&\frac{(-1)^m\pi^{2m-1}}{(2m-1)!}\bigg(2^{2m-1}\z{\hskip0.15mm}'\Big(1-2m,1-\frac x{2\pi}\Big)
-2^{2m-1}\z{\hskip0.15mm}'\Big(1-2m,\frac x{2\pi}\Big)\\
&-\z{\hskip0.15mm}'\Big(1-2m,1-\frac x\pi\Big)+\z{\hskip0.15mm}'\Big(1-2m,\frac x\pi\Big)
+\Big(\frac x\pi\Big)^{2m-1}\log2\\
&+2^{2m}\sum_{k=1}^{m}\binom{2m-1}{2k-1}\Big(\frac x{2\pi}\Big)^{2k-1}(1-2^{1-(2m-2k+1)})\z{\hskip0.15mm}'(2k-2m)\bigg).
\end{split}
\end{equation}

By differentiating \eqref{z-fun-eq} we can evaluate $\z\hs(-2n)$ for positive integers $n$.
So, the left-hand side is $-\z\hs(1-s)$, but on the right-hand side, we obtain a sum of four
terms and notice that in three of them, there appears $\cos\frac{\pi s}2$ whilst only one
contains $\sin\frac{\pi s}2$, i.e.
\begin{equation*}
-\frac{\pi\z(s)}{(2\pi)^{s}}\G(s)\sin\frac{\pi s}2.
\end{equation*}
With $s=2n+1$, all the terms except for the latter become zero, and we find
\begin{equation}\label{der-zeta-even}
\z\hs(-2n)=\frac{(-1)^n(2n)!\z(2n+1)}{2(2\pi)^{2n}},\quad n\in\mathbb N.
\end{equation}
Because of \eqref{der-zeta-even} and $\eta(s)=(1-2^{1-s})\z(s)$, for the last row of \eqref{subtr}, we have
$$
-(-1)^{m}\frac{(2m-1)!}{\pi^{2m-1}}\sum_{k=0}^{m-2}\frac{(-1)^{k}\eta(2m-2k-1)}{(2k+1)!}x^{2k+1}.
$$
Adding this way modified formula \eqref{subtr} to \eqref{log2}, after a rearrangement,
we arrive at \eqref{cf-sin-al}. \end{proof}

\begin{theorem} For $\a=2m-1$, the series \eqref{polylogarithm1-1-2} takes the closed form
\begin{equation}\label{cf-cos-al}
\begin{split}
\sum_{n=1}^\infty&\frac{(-1)^{n-1}\cos nx}{n^{2m-1}}
=\frac{(-1)^{m-1}\pi^{2m-2}}{(2m-2)!}\bigg(
2^{2m-2}\z{\hskip0.15mm}'\Big(2-2m,1-\frac x{2\pi}\Big)\\
&+2^{2m-2}\z{\hskip0.15mm}'\Big(-2m+2,\frac x{2\pi}\Big)
-\z{\hskip0.15mm}'\Big(2-2m,1-\frac x\pi\Big)
-\z{\hskip0.15mm}'\Big(2-2m,\frac x\pi\Big)\bigg).
\end{split}
\end{equation}
\end{theorem}

\begin{proof} Following a similar procedure as in the proof of the preceding theorem, we replace
$\a$ with $2m-1$ in \eqref{polylogarithm1-1-2}, and have
$$
\sum_{n=1}^\infty\frac{(-1)^{n-1}\cos nx}{n^{2m-1}}
=\sum_{k=0}^\infty\frac{(-1)^k\eta(2m-2k-1)}{(2k)!}\,x^{2k}.
$$
The function $\eta(s)$ is analytic in the whole complex plane, so to bring the latter series in closed form,
we express it first as a sum comprising the value of $\eta(1)$, obtained for $k=m-1$, and its remainder, i.e.
\begin{multline}\label{log-2}
\sum_{k=0}^{m-2}\frac{(-1)^k\eta(2m-2k-1)}{(2k)!}x^{2k}+\frac{(-1)^{m-1}x^{2m-2}}{(2m-2)!}\log2\\
+\sum_{k=m}^\infty\frac{(-1)^k\eta(2m-2k-1)}{(2k)!}x^{2k}.
\end{multline}
Further, we act in the same manner as in the previous proof. Relying on the relation
$\eta(s)=(1-2^{1-s})\z(s)$ and \eqref{z1-2n}, we determine the remainder in \eqref{log-2}
as a difference of the two series
\begin{multline*}
\sum_{k=m}^\infty\frac{(-1)^k\z(2m-2k-1)}{(2k)!}x^{2k}
-\sum_{k=m}^\infty\frac{(-1)^k\z(2m-2k-1)}{2^{2m-2k-2}(2k)!}x^{2k}\\
=2(-1)^{m-1}x^{2m-2}\bigg(\sum_{k=1}^\infty\frac{\z(2k)}{(2k)_{2m-1}}\Big(\frac x{2\pi}\Big)^{2k}
-\sum_{k=1}^\infty\frac{\z(2k)}{(2k)_{2m-1}}\Big(\frac x\pi\Big)^{2k}\bigg).
\end{multline*}

Applying again Theorem \ref{th-choi-srivastava}, where we take $n=2m-2$ and $a=1$, then in succession
set $t=\frac x{2\pi}$ and $t=\frac x{\pi}$, we obtain two formulas of the same structure. Subtracting
them gives rise to the sum of the right-hand side series in \eqref{log-2}, which is
\begin{equation}\label{subtr-2}
\begin{split}
&\frac{(-1)^{m-1}\pi^{2m-2}}{(2m-2)!}\bigg(2^{2m-2}\z{\hskip0.15mm}'\Big(2-2m,1-\frac x{2\pi}\Big)
+2^{2m-2}\z{\hskip0.15mm}'\Big(2-2m,\frac x{2\pi}\Big)\\
&-\z{\hskip0.15mm}'\Big(2-2m,1-\frac x\pi\Big)-\z{\hskip0.15mm}'\Big(2-2m,\frac x\pi\Big)
-\Big(\frac x\pi\Big)^{2m-2}\log2\\
&-(-1)^{m-1}\frac{(2m-2)!}{\pi^{2m-2}}\sum_{k=0}^{m-2}\frac{(-1)^{k}\eta(2m-2k-1)}{(2k)!}x^{2k}\bigg).
\end{split}
\end{equation}
Adding \eqref{subtr-2} to \eqref{log-2}, we obtain \eqref{cf-sin-al}. \end{proof}

By using 
the identity $\z(s)+\eta(s)=2\la(s)$, and summing up the first and second equation of \eqref{polylogarithm1}
with \eqref{polylogarithm1-2-1} and \eqref{polylogarithm1-1-2} respectively, we obtain
\begin{align}
\sum_{n=1}^\infty\frac{\sin(2n-1)x}{(2n-1)^\a}
&=\frac{\pi x^{\a-1}}{4\G(\a)\sin\frac\pi2\a}
+\sum_{k=0}^\infty\frac{(-1)^k\la(\a-2k-1)}{(2k+1)!}\,x^{2k+1},\label{polylogarithm3-1-2}\\
\sum_{n=1}^\infty\frac{\cos(2n-1)x}{(2n-1)^\a}
&=\frac{\pi x^{\a-1}}{4\G(\a)\cos\frac\pi2\a}
+\sum_{k=0}^\infty\frac{(-1)^k\la(\a-2k)}{(2k)!}\,x^{2k},\quad0<x<\pi,\notag
\end{align}

\begin{theorem} If $\a\to2m$ in the first formula of \eqref{polylogarithm3-1-2}, there holds
\begin{equation}\label{cf-sin-cos-odd-1}
\begin{split}
\sum_{n=1}^\infty&\frac{\sin(2n-1)x}{(2n-1)^{2m}}
=\frac{(-1)^m\pi^{2m-1}}{2(2m-1)!}\bigg(2^{2m}\z\hs\Big(1-2m,1-\frac x{2\pi}\Big)\\
&-2^{2m}\z\hs\Big(1-2m,\frac x{2\pi}\Big)-\z\hs\Big(1-2m,1-\frac x{\pi}\Big)+\z\hs\Big(1-2m,\frac x{\pi}\Big)\bigg).
\end{split}
\end{equation}
\end{theorem}

\begin{proof} Since $\sin\frac{2m\pi}2=0$, in the first formula of \eqref{polylogarithm3-1-2} for $\a=2m$ singularity is encountered, and for $k=m-1$ there is singularity of the lambda function at $1$, so we are not permitted to immediately replace $\a$ with $2m$ but have to take limit
\begin{equation}\label{sin2n-1-2m}
L_{2m}=\lim_{\a\to2m}\Big(\frac{\pi x^{\a-1}}{4\G(\a)\sin\frac\pi2\a}
+\frac{(-1)^{m-1}\la(\a-2m+1)}{(2m-1)!}\,x^{2m-1}\Big).
\end{equation}
Further, by bringing the fractions to the same denominator, we have
$$
\lim_{\a\to2m}\frac{\pi x^{\a-1}(2m-1)!+4(-1)^{m-1}x^{2m-1}\la(\a-2m+1)\G(\a)\sin\frac\pi2\a}{4(2m-1)!\G(\a)\sin\frac\pi2\a}.
$$
The denominator tends to zero as $\a$ tends to $2m$.

As for the numerator, knowing that $\la(s)=(1-2^{-s})\z(s)$, we have
\begin{align*}
&\lim_{\a\to2m}\Big(\pi x^{\a-1}(2m-1)!+4x^{2m-1}\G(\a)\la(\a-2m+1)(-1)^{m-1}\sin\frac\pi2\a\Big)\\
&=\pi x^{2m-1}(2m-1)!\Big(1-2\lim_{\a\to2m}(\a-2m)\la(\a-2m+1)\frac{\sin\frac\pi2(\a-2m)}{\frac\pi2(\a-2m)}\Big)=0,
\end{align*}
where, relying on the functional equation for the Riemann zeta function \eqref{z-fun-eq}, we made use of the limiting value
\begin{equation*}
\begin{split}
\lim_{s\to0}s\,\la(s+1)&=\lim_{s\to1}(s-1)\,\la(s)
=\lim_{s\to1}\frac{(2\pi)^s(s-1)(1-2^{-s})\z(1-s)}{2\hskip0.15mm\G(s)\cos\frac{s\pi}2}\\
&=-\lim_{s\to1}\frac{\pi^{s-1}\frac\pi2(1-s)\z(1-s)}{\G(s)\sin\frac\pi2(1-s)}
=-\z(0)=\tfrac12.
\end{split}
\end{equation*}

So, we can apply here L'Hospital's rule. For the limiting value of the first derivative of the denominator, we find
\begin{equation*}
4(2m-1)!\lim_{\a\to2m}\Big(\frac\pi2\cos\frac{\pi\a}{2}\G(\a)
+\sin\frac{\pi\a}{2}\G(\a)\psi(\a)\Big)=(-1)^m2\pi(2m-1)!^2,
\end{equation*}

Now, we take the limit of the first derivative of the numerator
\begin{align}
&\pi(2m-1)!x^{2m-1}\log x+(-1)^{m-1}x^{2m-1}\lim_{\a\to2m}\Big(4\la{\hskip0.3mm}'(\a-2m+1)\G(\a)\sin\frac{\pi\a}{2}\notag\\
&+4\la(\a-2m+1)\G(\a)\psi(\a)\sin\frac{\pi\a}2+2\pi\la(\a-2m+1)\G(\a)\cos\frac{\pi\a}2\Big).\label{expr}
\end{align}
First of all, using the relation $\la(s)=(1-2^{-s})\z(s)$, we find
$$
\frac{\la{\hskip0.3mm}'(s)}{\la(s)}=\frac{\z{\hskip0.3mm}'(s)}{\z(s)}+\frac{\log2}{2^s-1}.
$$
It is easy to see that in the neighbourhood of $s=1$, there holds
$$
\frac{\log2}{2^s-1}=\log2+O(s-1),
$$
and looking up in \cite[p.~23]{titchmarsh}, we read
$$
\frac{\z{\hskip0.3mm}'(s)}{\z(s)}=-\frac1{s-1}+\g+O(s-1).
$$
Thus there follows
$$
\la{\hskip0.3mm}'(\a-2m+1)=-\frac{\la(\a-2m+1)}{\a-2m}+(\g+\log2)\la(\a-2m+1)+O(\a-2m).
$$
By replacing this in \eqref{expr} omitting $O(\a-2m)$, there remains to calculate
\begin{multline*}
\lim_{\a\to2m}\G(\a)\la(\a-2m+1)\Big(-2\pi\frac{(-1)^m\sin\frac\pi2(\a-2m)}{\frac\pi2(\a-2m)}\\
+4\sin\frac{\pi\a}{2}(\psi(\a)+\g+\log2)+2\pi\cos\frac{\pi\a}2\Big).
\end{multline*}
Because of $\lim\limits_{\a\to2m}\cos\frac{\pi\a}2=\cos\pi m=(-1)^m$, we conclude that the limiting value of the sum of only the first and third terms is zero, and finding the limit reduces to
\begin{multline*}
\hskip-0.2cm\lim_{\a\to2m}\frac{2(-1)^m\pi\G(\a)\sin\frac\pi2(\a-2m)}{\frac\pi2(\a-2m)}(\a-2m)\la(\a-2m+1)(\psi(\a)+\g+\log2)\\
\hskip0.5cm=(-1)^m\pi(2m-1)!(\psi(2m)+\g+\log2)=(-1)^m\pi(2m-1)!(H_{2m-1}+\log2).
\end{multline*}
We have applied here the relation \eqref{psi-hn}.

Thus, taking account of all of this, the value of \eqref{sin2n-1-2m} is
$$
L_{2m}=\frac{(-1)^mx^{2m-1}(\log x-H_{2m-1}-\log2)}{2(2m-1)!},
$$
and the trigonometric series \eqref{cf-sin-cos-odd-1} can now be expressed as follows
\begin{align}
\sum_{n=1}^\infty&\frac{\sin(2n-1)x}{(2n-1)^{2m}}
=\frac{(-1)^{m}x^{2m-1}}{2(2m-1)!}\Big(\log\frac x2-H_{2m-1}\Big)\label{lim-value2}\\
&+\sum_{k=0}^{m-2}\frac{(-1)^k\la(2m-2k-1)}{(2k+1)!}\,x^{2k+1}
+\sum_{k=m}^\infty\frac{(-1)^k\la(2m-2k-1)}{(2k+1)!}\,x^{2k+1},\notag
\end{align}
and after shifting the summation index in the last series, making use of
the relation $\la(2m-2k-1)=(1-2^{-(2m-2k-1)})\z(2m-2k-1)$  and applying \eqref{z1-2n},
one can represent it as a difference of two series in terms of the zeta function
$$
(-1)^{m-1}x^{2m-1}\bigg(2\sum_{k=1}^\infty\frac{\z(2k)}{(2k)_{2m}}\,\Big(\frac x{2\pi}\Big)^{2k}
-\sum_{k=1}^\infty\frac{\z(2k)}{(2k)_{2m}}\,\Big(\frac x\pi\Big)^{2k}\bigg).
$$
We have already dealt with these series, the formula \eqref{diff-zeta-2k}, but here we express
this difference, as follows
\begin{align*}
&\hskip-0.3cm2(-1)^{m-1}x^{2m-1}\sum_{k=1}^\infty\frac{\z(2k)}{(2k)_{2m}}\Big(\frac x{2\pi}\Big)^{2k}
-(-1)^{m-1}x^{2m-1}\sum_{k=1}^\infty\frac{\z(2k)}{(2k)_{2m}}\Big(\frac x{\pi}\Big)^{2k}\\
&=\frac{(-1)^m\pi^{2m-1}}{2(2m-1)!}\bigg(2^{2m}\z\hs\Big(1-2m,1-\frac x{2\pi}\Big)-2^{2m}\z\hs\Big(-2m+1,1+\frac x{2\pi}\Big)\\
&+\Big(\frac x{\pi}\Big)^{2m-1}(H_{2m-1}-\log2\pi)-\z\hs\Big(1-2m,1-\frac x{\pi}\Big)+\z\hs\Big(1-2m,1+\frac x{\pi}\Big)\\
&-\frac{2(-1)^m(2m-1)!}{\pi^{2m-1}}\sum_{k=0}^{m-2}\frac{x^{2k+1}}{(2k+1)!}(-1)^{k}\la(2m-2k-1)\bigg).\notag
\end{align*}
Adding the right-hand side to \eqref{lim-value2}, after a rearrangement and using \eqref{hurwitz-a-a+1},
we obtain \eqref{cf-sin-cos-odd-1}. \end{proof}

\begin{theorem} If  $\a\to2m-1$ in the second formula of \eqref{polylogarithm3-1-2}, there holds
\begin{align}
\sum_{n=1}^\infty&\frac{\cos (2n-1)x}{(2n-1)^{2m-1}}
=\frac{(-1)^{m-1}\pi^{2m-2}}{2(2m-2)!}\bigg(2^{2m-1}\z{\hskip0.15mm}'\Big(2-2m,1-\frac x{2\pi}\Big)\label{cf-sin-cos-odd-2}\\
&+2^{2m-1}\z{\hskip0.15mm}'\Big(2-2m,1+\frac x{2\pi}\Big)-\z{\hskip0.15mm}'\Big(2-2m,1-\frac x\pi\Big)
-\z{\hskip0.15mm}'\Big(2-2m,1+\frac x\pi\Big)\bigg).\notag
\end{align}
\end{theorem}

\begin{proof}  After taking the limit $\a\to2m-1$ in the second formula of \eqref{polylogarithm3-1-2}, acting in the same manner as in the proof of the previous theorem, we first find
\begin{align}
\sum_{n=1}^\infty&\frac{\cos (2n-1)x}{(2n-1)^{2m-1}}
=\frac{(-1)^{m}x^{2m-2}}{2(2m-2)!}\Big(\log\frac x2-H_{2m-2}\Big)\label{lim-value3}\\
&+\sum_{k=0}^{m-2}\frac{(-1)^k\la(2m-2k-1)}{(2k)!}\,x^{2k}
+\sum_{k=m}^\infty\frac{(-1)^k\la(2m-2k-1)}{(2k)!}\,x^{2k}.\notag
\end{align}
We deal again with the remainder and rewrite it as follows
$$
(-1)^{m-1}x^{2m-2}\bigg(2\sum_{k=1}^\infty\frac{\z(2k)}{(2k)_{2m-1}}\,\Big(\frac x{2\pi}\Big)^{2k}
-\sum_{k=1}^\infty\frac{\z(2k)}{(2k)_{2m-1}}\,\Big(\frac x\pi\Big)^{2k}\bigg).
$$
Employing a similar procedure as in the case of the preceding theorem, we come to the closed-form formula
\eqref{cf-sin-cos-odd-2}. \end{proof}

Alternating series related to \eqref{polylogarithm3-1-2}
can be expressed as a power series involving Dirichlet's beta function
\begin{align}
\sum_{n=1}^\infty\frac{(-1)^{n-1}\sin(2n-1)x}{(2n-1)^\a}
&=\sum_{k=0}^\infty\frac{(-1)^k\b(\a-2k-1)}{(2k+1)!}x^{2k+1},\label{polylogarithm4}\\
\hskip-1cm\sum_{n=1}^\infty\frac{(-1)^{n-1}\cos(2n-1)x}{(2n-1)^\a}
&=\sum_{k=0}^\infty\frac{(-1)^k\b(\a-2k)}{(2k)!}x^{2k},\quad-\tfrac\pi2<x<\tfrac\pi2.\notag
\end{align}

\begin{theorem} If $\a$ is replaced with $2m-1$ in the first formula of \eqref{polylogarithm4},
we obtain the following closed form
\begin{equation}\label{swap-sin-cos-1}
\begin{split}
&\hskip-0.3cm\sum_{n=1}^\infty\frac{(-1)^{n-1}\sin(2n-1)x}{(2n-1)^{2m-1}}=\frac{(-1)^{m-1}(2\pi)^{2m-2}}{2(2m-2)!}
\Big(\z{\hskip0.15mm}'\Big(2-2m,\frac14-\frac x{2\pi}\Big)\\
&-\z{\hskip0.15mm}'\Big(2-2m,\frac34-\frac x{2\pi}\Big)
-\z{\hskip0.15mm}'\Big(2-2m,\frac14+\frac x{2\pi}\Big)+\z{\hskip0.15mm}'\Big(2-2m,\frac34+\frac x{2\pi}\Big)\Big),
\end{split}
\end{equation}
\end{theorem}

\begin{proof} These replacements are legitimate because we do not encounter singularities given the function $\b$ being analytic in the whole complex plane, knowing that $\b(1)=\arctg 1=\frac\pi4$ and by using \eqref{beta-ext}, we easily calculate
$$
\b(0)=\frac2\pi\sin\frac\pi2\G(1)\b(1)=\frac2\pi\cdot\frac\pi4=\frac12.
$$
Further, from \eqref{polylogarithm4}, we have
\begin{align*}
\sum_{n=1}^\infty\frac{(-1)^{n-1}\sin(2n-1)x}{(2n-1)^{2m-1}}
=\sum_{k=0}^\infty\frac{(-1)^k\b(2m-2k-2)}{(2k+1)!}x^{2k+1},
\end{align*}
splitting the right-hand side series in three, but writing it for brevity now as
\begin{equation*}
\sum_{k=0}^{m-2}\frac{(-x^2)^k x\b(2m-2k-2)}{(2k+1)!}-\frac{(-1)^{m}x^{2m-1}}{2(2m-1)!}
+\sum_{k=m}^\infty\frac{(-x^2)^k x\b(2m-2k-2)}{(2k+1)!}.
\end{equation*}
By shifting the summation index and making use of \eqref{beta-z}, the last series becomes
\begin{equation}\label{beta-series}
\sum_{k=m}^\infty\frac{(-x^2)^k x\b(2m-2k-2)}{(2k+1)!}=
(-1)^{m-1}x^{2m-2}\sum_{k=1}^\infty\frac{\b(2k+1)(\frac{2x}\pi)^{2k+1}}{(2k+1)_{2m-1}}.
\end{equation}
By expressing $\b(s)$ as a difference of two series
\begin{equation*}
\b(s)=\sum_{n=1}^\infty\frac{(-1)^{n-1}}{(2n-1)^s}=\sum_{n=0}^\infty\frac{1}{(4n+1)^s}-\sum_{n=0}^\infty\frac{1}{(4n+3)^s}
\end{equation*}
we can represent it through two values of Hurwitz's zeta function
\begin{equation}\label{beta-hurwitz}
\b(s)=4^{-s}\left(\z\big(s,\tfrac14\big)-\z\big(s,\tfrac34\big)\right).
\end{equation}
Taking account of \eqref{beta-hurwitz}, we further change \eqref{beta-series} to
\begin{equation}\label{dirichlet-beta}
(-1)^{m-1}x^{2m-2}\sum_{k=1}^\infty
\frac{\z\big(2k+1,\frac14\big)-\z\big(2k+1,\frac34\big)}{(2k+1)_{2m-1}}
\Big(\frac x{2\pi}\Big)^{2k+1}\end{equation}
That means we are dealing with two series over the Hurwitz zeta functions and set $n=2m-2$
and $t=\frac x{2\pi}$ in \eqref{choi-srivastava}, considering, apart from $t$, the same formula for $-t$ as well.

Replacing $a$ in succession with $\frac14$ and $\frac34$ in \eqref{choi-srivastava}, then subtracting these equalities, we obtain \eqref{dirichlet-beta} on the left-hand side. The right-hand side consists of a sum of four terms, the first one being
\begin{align*}
&\frac{(-1)^{m-1}(2\pi)^{2m-2}}{2(2m-2)!}
\Big(\z{\hskip0.15mm}'\Big(-2m+2,\frac14-\frac x{2\pi}\Big)-\z{\hskip0.15mm}'\Big(-{2m+2},\frac34-\frac x{2\pi}\Big)\\
&-\z{\hskip0.15mm}'\Big(-2m+2,\frac14+\frac x{2\pi}\Big)+\z{\hskip0.15mm}'\Big(-{2m+2},\frac34+\frac x{2\pi}\Big)\Big)
\end{align*}
then of two sums
\begin{equation}\label{two-sums}
\begin{split}
&\hskip-0.2cm(-1)^{m}\Big(\frac\pi2\Big)^{2m-2}\sum_{k=1}^{m-1}\!\Big(\frac {2x}{\pi}\Big)^{2k-1}\!\!
\frac{H_{2m-2}-H_{2m-2k-1}}{(2m-2k-1)!(2k-1)!}\b(2k-2m+1)\\
&\hskip-0.2cm+\frac{(-1)^{m}(2\pi)^{2m-2}}{(2m-2)!}\sum_{k=1}^{m-1}\!\Big(\frac x{2\pi}\Big)^{2k-1}\!
\binom{2m-2}{2k-1}4^{2k-2m+1}\b'(2k-2m+1)
\end{split}
\end{equation}
and finally of
\begin{equation}\label{psi-14-34}
\big(\psi(\tfrac14)-\psi(\tfrac34)\big)\frac{(-1)^{m-1}(2\pi)^{2m-2}}{(2m-1)!}\Big(\frac x{2\pi}\Big)^{2m-1}.
\end{equation}
For $k=1,\dots,m-1$ the expressions $2k-2m+1$ present negative odd integers, which means $\b(2k-2m+1)=0$,
so the first sum in \eqref{two-sums} equals zero.

Differentiating the relation \eqref{beta-z} at $z=2k-2m+1$, for $1\leqslant k\leqslant m-1$, yields
$$
\b'(2k-2m+1)=-\Big(\frac\pi2\Big)^{2k-2m+1}(-1)^{k-m}\G(2m-2k)\b(2m-2k),
$$
whereby the second sum in \eqref{two-sums} becomes
$$
-\sum_{k=1}^{m-1}\frac{(-1)^{k-1}\b(2m-2k)}{(2k-1)!}
=-\sum_{k=0}^{m-2}\frac{(-1)^{k}\b(2m-2k-2)}{(2k+1)!}.
$$

The expression \eqref{psi-14-34} is transformed into
$$
-4\b(1)\frac{(-1)^{m-1}(2\pi)^{2m-2}}{(2m-1)!}\Big(\frac x{2\pi}\Big)^{2m-1}
=4\cdot\frac\pi4\frac{(-1)^{m}x^{2m-1}}{2\pi(2m-1)!}=\frac{(-1)^{m}x^{2m-1}}{2(2m-1)!},
$$
and this is obtained by setting $n=1$ in the relation
\begin{equation*}
\b(n)=\frac{(-1)^n}{2^{2n}(2n-1)!}\big(\psi^{(n-1)}\big({\tfrac14}\big)-\psi^{(n-1)}\big(\tfrac34\big)\big),\quad n\in\mathbb N,
\end{equation*}
which in turn we get form \eqref{beta-hurwitz} by relying on the identity \cite{bateman}
\begin{equation*}
\psi^{(n)}(a)=(-1)^{n-1}n!\z(n+1,a).
\end{equation*}

Collecting all the cases, we arrive at the closed-form formula \eqref{swap-sin-cos-1}. \end{proof}

\begin{theorem} If $\a$ is replaced with $2m$ in the second formula of \eqref{polylogarithm4}, we  obtain the following closed forms
\begin{equation}\label{swap-sin-cos-2}
\begin{split}
&\hskip-0.3cm\sum_{n=1}^\infty\frac{(-1)^{n-1}\cos(2n-1)x}{(2n-1)^{2m}}=\frac{(-1)^{m-1}(2\pi)^{2m-1}}{2(2m-1)!}
\Big(\z{\hskip0.15mm}'\Big(1-2m,\frac14-\frac x{2\pi}\Big)\\
&-\z{\hskip0.15mm}'\Big(1-2m,\frac34-\frac x{2\pi}\Big)
+\z{\hskip0.15mm}'\Big(1-2m,\frac14+\frac x{2\pi}\Big)-\z{\hskip0.15mm}'\Big(1-2m,\frac34+\frac x{2\pi}\Big)\Big).
\end{split}
\end{equation}
\end{theorem}

\begin{proof} Acting in the same manner as in deducing the formula \eqref{swap-sin-cos-1},
we come to the formula \eqref{swap-sin-cos-2}. \end{proof}

\section{Appendix}

General closed-form formula type based on the formulas \eqref{lim-value1}, \eqref{lim-value1-2}, \eqref{cf-sin-al}, \eqref{cf-cos-al}, \eqref{cf-sin-cos-odd-1}, \eqref{cf-sin-cos-odd-2}, \eqref{swap-sin-cos-1} and \eqref{swap-sin-cos-2} is
\begin{align*}
\sum_{n=1}^\infty&\frac{(s)^{n-1}f((an-b)x)}{(an-b)^{2m+p-1}}
=\frac{(-1)^{m+p-1}\pi^{2m+p-2}2^{2m+r-2}}{(2m+p-2)!}\times\\
&\bigg(2^{2k-2}\z{\hskip0.15mm}'\big(2-p-2m,q-\frac x{2\pi}\big)+2^{2k-2}\d\z{\hskip0.15mm}'\big(2-p-2m,1-q+\frac x{2\pi}\Big)\\
&-j\Big(\z{\hskip0.15mm}'\big(2-p-2m,1-q-\frac x{2c\pi}\big)+\d\z{\hskip0.15mm}'\big(2-p-2m,q+\frac x{2c\pi}\big)\Big)\bigg),
\end{align*}
and one can obtain all its particular cases from Table II by choosing the corresponding parameters.
{\tabcolsep 4pt
\begin{center}
\begin{tabular}{|c|c|c|c|c|c|c|c|c|c|c|c|} \multicolumn{12}{c}
{Table II: Closed-form cases, $\a=2m+p-1$}\\ \hline
 $a$ & $b$ & $s$ & $f$ & $p$ & $r$ & $c$ & $\d$ & $q$ & $k$ & $j$ & Convergence region\\ \hline
     &   & & $\sin$ & 1 & 1 &  & -1 & $1$ & 1 & 0 & \\ \cline{4-11}
     &   & \raisebox{1.5ex}[0pt]{$1$} & $\cos$ & 0 & 0 &  & $1$ & $1$ & 1 & 0 & \raisebox{1.5ex}[0pt]{$0<x<2\pi$}\\ \cline{3-12}
   &   &  & $\sin$ & $1$ & $2-2m$ & -$\frac12$ & -1 & $1$ & $m+\frac12$ & -1 & \\ \cline{4-11}
 \raisebox{4.5ex}[0pt]{$1$} & \raisebox{4.5ex}[0pt]{$0$} & \raisebox{1.5ex}[0pt]{-1} & $\cos$ & $0$ & $2-2m$ & -$\frac12$ & $1$ & $1$ & $m$ & 1 & \raisebox{1.5ex}[0pt]{$-\pi<x<\pi$}\\ \hline
 &   &    & $\sin$ & $1$ & $1-2m$ & -$\frac12$ & -1 & $1$ & $m+1$ & -1 & \\ \cline{4-11}
   &   &  \raisebox{1.5ex}[0pt]{1}& $\cos$ & 0 & $1-2m$ & -$\frac12$ & 1 & 1 & $m+\frac12$ & 1 & \raisebox{1.5ex}[0pt]{$0<x<\pi$}\\ \cline{3-12}
   &   &    & $\sin$ & 0 & -1 & 1 & 1 & $\frac14$ & $1$ & 1 & \\ \cline{4-11}
 \raisebox{4.5ex}[0pt]{$2$} & \raisebox{4.5ex}[0pt]{$1$} & \raisebox{1.5ex}[0pt]{-1}& $\cos$ & 1 & 0 & 1 & -1 & $\frac14$ & 1 & -1 & \raisebox{1.5ex}[0pt]{$-\frac\pi 2<x<\frac\pi2$}\\ \hline
\end{tabular}
\end{center}}

\end{document}